\documentclass[11pt]{article}
\usepackage[T1]{fontenc}
\usepackage{lmodern}
\usepackage{amsmath,amssymb,amsthm,mathtools}
\usepackage{array,booktabs}
\usepackage{graphicx}
\usepackage{float}
\usepackage{placeins}
\usepackage{microtype}
\usepackage[margin=1in]{geometry}
\usepackage{authblk}
\usepackage[hidelinks]{hyperref}
\usepackage{url}

\newtheorem{theorem}{Theorem}
\newtheorem{lemma}[theorem]{Lemma}
\newtheorem{proposition}[theorem]{Proposition}
\newtheorem{corollary}[theorem]{Corollary}

\newcommand{\R}{\mathbb R}

\newcommand{\norm}[1]{\lVert #1\rVert}

\title{A counterexample to global convergence of classical DFP
under the standard strong Wolfe conditions}

\author[1]{Benqi Liu}
\author[2]{Zichen Wang}
\author[1]{Zaiwen Wen}
\author[3,4]{Liwei Zhang}
\author[5]{Yaxiang Yuan}
\affil[1]{Beijing International Center for Mathematical Research,
Peking University, Beijing 100871, China;
\texttt{bqliu@pku.edu.cn}; \texttt{wenzw@pku.edu.cn}}
\affil[2]{School of Mathematical Sciences, Peking University,
Beijing 100871, China;
\texttt{zichenwang25@stu.pku.edu.cn}}
\affil[3]{National Frontiers Science Center for Industrial Intelligence and
Systems Optimization, Northeastern University, Shenyang 110819, China.}
\affil[4]{Key Laboratory of Data Analytics and Optimization for Smart Industry
(Northeastern University), Ministry of Education, Shenyang 110819, China;
\texttt{zhanglw@mail.neu.edu.cn}}
\affil[5]{Academy of Mathematics and Systems Science,
Chinese Academy of Sciences, Beijing 100190, China;
\texttt{yyx@lsec.cc.ac.cn}}
\date{}

\begin{document}
\maketitle

\begin{abstract}
A long-standing open question in quasi-Newton optimization asks whether the
classical Davidon--Fletcher--Powell (DFP) method converges globally on uniformly
convex objectives when all accepted steps satisfy the standard weak Wolfe
conditions.  We show that the answer is no, even under the standard strong
Wolfe conditions.  Fix $0<c_1<2/3$ and $2/3\le c_2<1$.  We construct a
function $f\in C^2(\R^2)$ such that
$\frac12 I\preceq\nabla^2 f(x)\preceq\frac32 I$ for all $x\in\R^2$.  We also
choose a fixed positive definite initial inverse Hessian approximation and
a sequence of positive step lengths.  The classical DFP iteration is well
defined, and all
accepted steps satisfy the standard strong Wolfe conditions, but
$\|\nabla f(x_k)\|$ converges to a positive constant.  The global Hessian
condition number is at most three.  The construction uses an
alternating two-step DFP sequence near a one-dimensional invariant center
manifold.  Along this sequence, the smaller eigenvalue of the inverse Hessian
approximation tends to zero.  The changes in the gradient norm between cycle
starts are summable, but the total rotation of the associated eigenvectors is
unbounded.  The accumulation points of the DFP sequence form a circle.  A
uniform separation bound allows us to interpolate the prescribed function
values and gradients.  We add smooth functions with pairwise disjoint supports
to a quadratic and keep the global Hessian bounds.  An affine change of
variables gives an identity-initialized example with problem-dependent Hessian
bounds.  An orthogonal direct sum extends the result to every dimension
$n\ge2$.
\end{abstract}



\section{Introduction}
\label{sec:introduction}

Does the classical Davidon--Fletcher--Powell (DFP) method always converge
on every uniformly convex objective when each accepted step satisfies the
standard weak Wolfe conditions?  This is a long-standing open question in the
global convergence theory of quasi-Newton methods.  Surveys by Nocedal,
Fletcher, and Yuan record the question
\cite{Nocedal1992,Fletcher1994,Yuan1999}.  A 2026 paper still describes it as
open \cite{YuanZhouPham2026}.  That paper proves convergence for a projected
and corrected DFP method, not for the classical method.  The open question
concerns the unmodified DFP method, an arbitrary positive definite initial
inverse Hessian approximation, and every sequence of step lengths that
satisfies the weak Wolfe conditions.  DFP is one of the earliest
variable-metric methods for unconstrained optimization
\cite{Davidon1991,FletcherPowell1963}.  The curvature inequality $s_k^Ty_k>0$
keeps the inverse Hessian approximation positive definite after each update.
The local convergence theory of DFP is well developed.  The main difficulty
is global.  Powell's convergence theorem for uniformly convex objectives uses
exact line search \cite{Powell1971}.  Exact line search gives the orthogonality
identity $g_{k+1}^Ts_k=0$.  This identity removes a positive term from the DFP
estimates, but the term remains under an inexact Wolfe line search.  Powell's
inexact-line-search convergence theorem for the
Broyden--Fletcher--Goldfarb--Shanno (BFGS) method \cite{Powell1976} does not
cover DFP.

Existing global convergence results either exclude DFP or add assumptions to
the two standard Wolfe inequalities.  Byrd, Nocedal, and Yuan proved a
self-correction property for a restricted convex Broyden class, but they
excluded DFP \cite{ByrdNocedalYuan1987}.  Yuan's DFP results require further
properties of the generated sequence.  These include eventual monotonicity of
the gradient norm and finite total step length \cite{Yuan1995}.  Xu assumes
eventual monotonicity of an iteration-dependent scalar and also requires
$c_2<(m/M)^3$ under $mI\preceq\nabla^2f\preceq MI$ \cite{Xu1997}.  Other
results use metric-dependent Wolfe parameters or special inexact line searches
\cite{Pu2002,LiuJingHan2002}.  Pu and Yu also assume that the DFP iterate
sequence converges \cite{PuYu1990}.  In two dimensions, Powell proved
$\liminf_{k\to\infty}\|g_k\|=0$ when each line search returns the first local
minimizer on the search line \cite{Powell2000}.  This rule again gives
$g_{k+1}^Ts_k=0$.  None of these results proves convergence of classical DFP
under only the two standard Wolfe inequalities.

This paper gives a negative answer.  For every $0<c_1<2/3$ and
$2/3\le c_2<1$, Theorem~\ref{thm:main} constructs a function
$f\in C^2(\R^2)$ satisfying $\frac12I\preceq\nabla^2f\preceq\frac32I$ and a
classical DFP sequence whose gradient norms converge to a positive constant.
Every accepted step satisfies the standard strong Wolfe conditions, so the
failure persists under a stronger curvature condition than the one posed in
the original question.  The Hessian condition number is uniformly bounded by
three.  The line-search requirement concerns only the accepted step lengths:
they satisfy the Wolfe inequalities at their endpoints.  We do not claim that
a specific line-search implementation, with its own trial steps and
interpolation rules, selects these steps.  The initial inverse Hessian
approximation in the main construction is positive definite, but it is not the
identity.  Corollary~\ref{cor:identity-initialization} uses an affine change of
variables to obtain $H_0=I$.  The transformed objective still has a uniformly
positive definite and globally bounded Hessian, but its bounds depend on the
original $H_0$.  The theorem gives the fixed Hessian bounds.  The corollary
gives identity initialization with problem-dependent bounds.

The construction has two main tasks.  First, a finite sequence of nearly
orthogonal DFP search directions does not prove nonconvergence, even if it is
very long.  We construct an infinite two-step DFP sequence in which the
smaller eigenvalue of $H_k$ tends to zero.
The changes in the gradient norm between successive cycle starts form a
summable sequence.  The sum of the corresponding eigenvector rotation angles
diverges.  The remaining asymptotic estimates then show that the accumulation
points form a circle.  Second, the prescribed secant pairs must come from one
$C^2$, uniformly convex objective on $\R^2$.  A uniform separation estimate
allows us to interpolate the prescribed function values and gradients by
adding smooth functions with pairwise disjoint supports.  This construction
keeps the global Hessian bounds.  The sequence does not satisfy the extra
assumptions in the known convergence results.  Its gradient norm decreases on
the first step of each cycle and increases on the second.  Its total step
length is infinite, and the chosen range of $c_2$ violates Xu's extra
restriction.  We use neither exact line search nor a search for the first
local minimizer on the search line.  Section~\ref{sec:known-results} checks
these facts after the construction is complete.

The classical DFP question and the main result are stated in
Section~\ref{sec:statement}.  Section~\ref{sec:strategy} gives the basic DFP
identities and an outline of the construction.  We develop the nonconvergent
sequence and its asymptotic behavior in
Section~\ref{sec:alternating-sequence}.  In
Section~\ref{sec:global-objective}, we build one $C^2$, uniformly convex
objective on $\R^2$ that matches the prescribed function values and gradients.
Section~\ref{sec:completion} completes the
proof.  It verifies the strong Wolfe conditions, proves nonconvergence, gives
the higher-dimensional and identity-initialized versions, and compares the
example with known sufficient conditions.  The numerical experiments appear
in Section~\ref{sec:numerics}.  Section~\ref{sec:conclusion} closes the paper,
and Appendix~\ref{app:algebra} checks the algebra of the two-step DFP
recurrence.

\section{The classical DFP question and its negative answer}
\label{sec:statement}

For $g_k=\nabla f(x_k)$, the classical DFP method written with an inverse
Hessian approximation is
\begin{align*}
d_k&=-H_kg_k,&s_k&=\alpha_kd_k,&x_{k+1}&=x_k+s_k,\\
y_k&=g_{k+1}-g_k,&
H_{k+1}&=H_k-
\frac{H_ky_ky_k^TH_k}{y_k^TH_ky_k}
+\frac{s_ks_k^T}{s_k^Ty_k}.
\end{align*}
The standard weak Wolfe conditions \cite{Wolfe1969} are
\begin{align}
f(x_{k+1})&\le f(x_k)+c_1g_k^Ts_k,\label{eq:armijo}\\
g_{k+1}^Ts_k&\ge c_2g_k^Ts_k,
\label{eq:weak-curvature}
\end{align}
where $0<c_1<c_2<1$.
The standard strong Wolfe conditions use \eqref{eq:armijo} and replace
\eqref{eq:weak-curvature} by
\begin{equation}\label{eq:strong-curvature}
|g_{k+1}^Ts_k|\le c_2|g_k^Ts_k|.
\end{equation}
In this paper, a line search accepts a step when its positive step length
satisfies the stated inequalities at $x_k$ and $x_{k+1}$.  We do not tie this
definition to a particular bracketing, interpolation, or zoom procedure.  If
$H_k\succ0$ and $g_k\ne0$, then
$g_k^Td_k<0$.  Either curvature condition
implies
\[
s_k^Ty_k=(g_{k+1}-g_k)^Ts_k
\ge (c_2-1)g_k^Ts_k>0,
\]
so the DFP denominators are positive and $H_{k+1}\succ0$.  In our
construction, positivity also follows directly from the prescribed secant
pairs.

The classical open question is the following.  Suppose
$f\in C^2(\R^n)$ satisfies
$mI\preceq\nabla^2f(x)\preceq MI$ on the relevant level set for constants
$0<m\le M<\infty$, let $H_0\succ0$ be arbitrary, and let every accepted
positive step satisfy \eqref{eq:armijo}--\eqref{eq:weak-curvature}.  Must the
unmodified DFP iterates satisfy
\begin{equation}\label{eq:open-question}
\|\nabla f(x_k)\|\longrightarrow0
\end{equation}
for every sequence of step lengths that satisfies these conditions?  A
uniformly convex objective has a unique minimizer, so
\eqref{eq:open-question} is equivalent to
convergence of the iterates to that minimizer.  The statement covers every
initial inverse Hessian approximation and every accepted weak Wolfe step.  A
counterexample needs only one uniformly convex objective, one positive
definite initial matrix, and one infinite sequence of valid steps for which
\eqref{eq:open-question} fails.  The main theorem gives such an example.  Its
accepted steps even satisfy the standard strong Wolfe conditions.

\begin{theorem}[Main theorem]\label{thm:main}
For every pair of constants satisfying
\begin{equation}\label{eq:wolfe-range}
0<c_1<\frac23,\qquad \frac23\le c_2<1,
\end{equation}
there exist $f\in C^2(\R^2)$, $x_0\in\R^2$, $H_0\succ0$, and positive step
lengths $\alpha_k$ such that
\[
\frac12I\preceq\nabla^2f(x)\preceq\frac32I
\qquad(x\in\R^2),
\]
all classical DFP iterates are well defined, every step satisfies
\eqref{eq:armijo} and \eqref{eq:strong-curvature} with these constants, but
there is a constant $G_\infty>0$ such that
\[
\|\nabla f(x_k)\|\longrightarrow G_\infty.
\]
The same conclusion holds in every dimension $n\ge2$ by an orthogonal direct
sum.
\end{theorem}

\begin{proof}
Section~\ref{sec:alternating-sequence} constructs an infinite
two-dimensional sequence satisfying the classical DFP search-direction, secant, and update identities.  It also proves that the gradient norms converge to a positive constant.  Section~\ref{sec:global-objective} establishes a uniform separation bound and realizes the prescribed endpoint
data by a globally $C^2$ function satisfying
$\frac12 I\preceq\nabla^2 f\preceq\frac32 I$.
Section~\ref{sec:completion} verifies the Armijo and standard strong Wolfe conditions and extends the construction to every dimension $n\ge2$.
The detailed arguments are given in those sections.
\end{proof}

The Hessian bounds in Theorem~\ref{thm:main} give a global condition number of
at most three.  The threshold $c_2=2/3$ is exact for this construction: on the
two iterations of each cycle, the ratios
$|g_{k+1}^Ts_k|/|g_k^Ts_k|$ equal $1/3$ and $2/3$.  The restriction $c_1<2/3$
comes from the smaller of the two limiting normalized objective decreases in
the Armijo condition.  The main theorem does not require the standard
initialization $H_0=I$.  This raises a natural question: does the
nonconvergence depend on the use of a nonidentity initial matrix?
The affine normalization in the following corollary shows that it does not.

\begin{corollary}[Identity initialization]\label{cor:identity-initialization}
Fix any pair $(c_1,c_2)$ satisfying \eqref{eq:wolfe-range}.  There exist
$\widetilde f\in C^2(\R^2)$, $z_0\in\R^2$, positive step lengths $\alpha_k$,
and constants $0<\widetilde m\le \widetilde M<\infty$ such that
\[
\widetilde m I\preceq\nabla^2\widetilde f(z)
\preceq\widetilde M I\qquad(z\in\R^2).
\]
For this objective, classical DFP initialized with $\widetilde H_0=I$ is well
defined.  Every accepted step satisfies \eqref{eq:armijo} and
\eqref{eq:strong-curvature}, but
\[
\liminf_{k\to\infty}\|\nabla\widetilde f(z_k)\|>0.
\]
The same conclusion holds in every dimension $n\ge2$.
\end{corollary}

Corollary~\ref{cor:identity-initialization} gives identity initialization, but
its Hessian bounds depend on the problem.  It does not claim that $H_0=I$ and
the fixed Hessian bounds $[\frac12 I,\frac32 I]$ hold at the same time.  The
proof uses an affine change of variables and appears in
Section~\ref{sec:completion}, after the two-dimensional example is complete.

\section{Overview of the construction}
\label{sec:strategy}

We first define sequences $(x_k,g_k,H_k)$ that satisfy the DFP
search-direction, secant, and update identities.  We then construct a globally
$C^2$ uniformly convex function whose values and gradients agree with these
data.  Separating the two steps makes the DFP recurrence explicit, but it also
creates an interpolation problem: the prescribed gradient differences must be
the secant vectors of a single objective.

\subsection{One DFP update}

We begin with the DFP identities used to generate the algebraic sequence.  The
calculation is independent of dimension.  Let $n\ge1$, let
$H,A\in\R^{n\times n}$ be symmetric positive definite matrices, let
$g\in\R^n\setminus\{0\}$, and let $\tau>0$.  Set
\begin{equation}\label{eq:abstract-step}
\alpha=\tau\frac{g^THg}{(Hg)^TA(Hg)},\qquad
s=-\alpha Hg,\qquad y=As,\qquad g_+=g+y,
\end{equation}
and apply the inverse form of the DFP update to $(H,s,y)$.  With
\[
q=-g^Ts>0,\qquad t=s^Ty,
\]
the choice of $\alpha$ gives the identity
\begin{equation}\label{eq:line-ratio}
\frac tq=\tau.
\end{equation}
Here $q=\alpha g^THg$ and
$t=\alpha^2(Hg)^TA(Hg)=\tau\alpha g^THg$.

For the two-step construction, it is useful to eliminate the step length
$\alpha$ from the DFP formulas and keep only the secant matrix $A$ and the
ratio $\tau$.  The following proposition gives the resulting gradient and
matrix updates.  We will use them to combine the two iterations in each cycle.

\begin{proposition}[DFP update identities]\label{prop:one-step}
Let $v=Hg$, $\delta=g^THg$, $w=Av$, $\beta=v^TAv$, and
$\gamma=w^THw$.  The step in \eqref{eq:abstract-step} satisfies
\begin{align}
g_+&=g-\tau\frac{\delta}{\beta}w,
\label{eq:one-step-g}\\
H_+&=H-\frac{Hww^TH}{\gamma}+\frac{vv^T}{\beta}.
\label{eq:one-step-H}
\end{align}
In particular, the matrix update is independent of $\tau$, and every
denominator in \eqref{eq:one-step-g}--\eqref{eq:one-step-H} is positive.
\end{proposition}

\begin{proof}
Equation \eqref{eq:abstract-step} gives $s=-\alpha v$ and
$y=-\alpha w$.  We have $s^Ty=\alpha^2\beta$ and
$y^THy=\alpha^2\gamma$; the factor $\alpha^2$ cancels in both rank-one DFP
terms.  This proves \eqref{eq:one-step-H}, while $g_+=g+y$ gives
\eqref{eq:one-step-g}.  Positivity follows from $H\succ0$, $A\succ0$, and
$v\ne0$.
\end{proof}

\subsection{Outline of the proof}

Before constructing the sequence, we introduce the notation used
throughout the proof and summarize its three main stages.  We use the
convention $e^\perp=(-e_2,e_1)^T$ for
$e=(e_1,e_2)^T$, so $(e,e^\perp)$ is a positively oriented orthonormal
basis.  Vector norms are Euclidean, and matrix norms are spectral norms.  With
this convention, five quantities describe the two-step sequence.  The table summarizes their
roles and asymptotic behavior.
\begin{center}
\small
\begin{tabular}{@{}lll@{}}
\toprule
quantity & meaning & behavior along cycle $j$\\
\midrule
$\epsilon_j$ & small parameter & $\asymp j^{-1/3}$\\
$r_j=\epsilon_j^2$ & order of step lengths and rotation angles & $\asymp j^{-2/3}$\\
$G_j$ & gradient component $e_j^Tg_{2j}$ & converges to a positive limit\\
$\phi_j$ & angle for the smaller eigenvalue
  & $\sum_j|\phi_{j+1}-\phi_j|=\infty$\\
$C_k=x_k-g_k$ & minimizer of the reference quadratic & converges in $\R^2$\\
\bottomrule
\end{tabular}
\end{center}
The smaller eigenvalue of $H_{2j}$ is of order $\epsilon_j^4$.  A pair of DFP
iterations changes $G_j$ by order $\epsilon_j^4$ but rotates the orthonormal
eigenbasis of $H_{2j}$ by order $\epsilon_j^2$.  The changes in $G_j$ are
summable.  The rotation angles are not.

The proof has three parts.  Section~\ref{sec:alternating-sequence} chooses two
positive definite matrices that define the secant pairs and uses
Proposition~\ref{prop:one-step} to generate an alternating
two-step DFP sequence.  The center manifold theorem gives an infinite
sequence on which $G_j$ has a positive limit and the iterates approach a
circle.  Section~\ref{sec:global-objective} proves a uniform separation
estimate and interpolates the prescribed values and gradients by
smooth functions with pairwise disjoint supports.  The matrices used to define
the secant pairs need not be Hessians of the final objective; after
interpolation, the actual gradient differences are exactly the prescribed
secants.  Section~\ref{sec:completion} verifies the Armijo and strong Wolfe
conditions, proves nonconvergence, and gives the identity-initialized and
higher-dimensional versions.

\section{Construction of a nonconvergent DFP sequence}
\label{sec:alternating-sequence}

We now construct the sequences $(x_k,g_k,H_k)$ before defining the objective.
Unlike Proposition~\ref{prop:one-step}, which is valid in any dimension, the
construction in this and the next section is entirely two-dimensional.  All
vectors belong to $\R^2$ and all matrices are $2\times2$ until the direct-sum
extension in Section~\ref{sec:completion}.  The sequences satisfy the
search-direction, secant, and update equations of classical DFP.  The analysis
below shows that $\|g_k\|$ remains bounded away from zero.
Section~\ref{sec:global-objective} constructs one uniformly convex function on
$\R^2$ that generates these data.  Cycle $j$ consists of DFP iterations
$2j$ and $2j+1$.

\subsection{Coordinate representation and choice of secant pairs}

We first choose coordinates that isolate the smaller eigenvalue of $H$ and
the component of $g$ along its eigenvector.  These coordinates put the two
prescribed secant pairs in a common form and allow us to analyze their
combined effect over one cycle.  At the beginning of a two-step cycle, let
$R=(e,e^\perp)$ be the oriented
orthogonal matrix whose first column is a unit eigenvector corresponding to
the smaller eigenvalue of $H$, and write
\begin{equation}\label{eq:coordinate-representation}
H=R\begin{pmatrix}hpr^2&0\\0&h\end{pmatrix}R^T,
\qquad
g=GR\binom{1}{pr},
\end{equation}
where $r,p,h,G>0$.  Here $r$ determines the smaller eigenvalue and the second
gradient component.  The scalar $G$ is the component of $g$ along the
corresponding eigenvector.  The pair $(p,h)$ describes the remaining normalized
components of $H$ and $g$.

We use the following data for the two DFP iterations in each cycle.  The ratio
$s^Ty/(-g^Ts)$ determines the directional derivative after an accepted step.
The off-diagonal terms have opposite signs, and this rotates the eigenvectors
of the inverse Hessian approximation.  At $\epsilon=0$, the first iteration
doubles $r$ and changes $p$ from $2$ to $1/2$.  The second restores both
quantities.  The terms of order $\epsilon$ produce the net decrease in $r$
that we study below.  Set $r=\epsilon^2$.  In the current oriented eigenbasis,
use the following matrices and parameters during one cycle:
\begin{equation}\label{eq:alternating-secant-data}
A_1(\epsilon)=
\begin{pmatrix}1&\epsilon\\ \epsilon&1\end{pmatrix},
\quad \tau_1=\frac23;
\qquad
A_2(\epsilon)=
\begin{pmatrix}1&-2\epsilon\\ -2\epsilon&1\end{pmatrix},
\quad \tau_2=\frac13.
\end{equation}
We restrict the initial parameter to
$0<\epsilon\le\epsilon_0<1/4$.  Both matrices are then positive definite, and
their spectra lie in $[1/2,3/2]$.  The same $\epsilon$ is used in both
iterations.  After each iteration, we diagonalize the updated inverse Hessian
approximation.  We choose the sign of the eigenvector for the smaller
eigenvalue so that its inner product with the new gradient is positive.  We
then use the corresponding coordinates in
\eqref{eq:coordinate-representation}.

For direct verification of the two iterations, we now record the current
gradient, the Hessian approximation, and the step length at each index.  This
also fixes the relation between the inverse-Hessian notation used in the
algorithm and the Hessian notation common in quasi-Newton analysis.  Define
\[
B_k:=H_k^{-1}.
\]
At the beginning of cycle $j$, write the quantities in
\eqref{eq:coordinate-representation} as
$R_j,G_j,p_j,h_j,r_j$, where $r_j=\epsilon_j^2$, and set
\[
\widehat g_j=\binom{1}{p_jr_j},\qquad
\widehat w_j=\binom{r_j+\epsilon_j}{1+\epsilon_jr_j},\qquad
D_j^{(1)}=1+2\epsilon_jr_j+r_j^2.
\]
The data at the even iteration are
\begin{equation}\label{eq:even-iteration-data}
\begin{aligned}
g_{2j}&=G_jR_j\widehat g_j,\\
B_{2j}&=R_j
\begin{pmatrix}
(h_jp_jr_j^2)^{-1}&0\\[1mm]
0&h_j^{-1}
\end{pmatrix}R_j^T,\\
\alpha_{2j}
&=\frac{2(p_j+1)}{3h_jp_jD_j^{(1)}}.
\end{aligned}
\end{equation}
The first update therefore gives, still in the basis $R_j$,
\begin{align}
g_{2j+1}
&=G_jR_j\left(
\widehat g_j-
\frac{2(p_j+1)r_j}{3D_j^{(1)}}\widehat w_j
\right),
\label{eq:first-updated-gradient}\\
B_{2j+1}
&=R_j\left\{
\frac1{h_j}
\begin{pmatrix}(p_jr_j^2)^{-1}&0\\0&1\end{pmatrix}
-\frac{\widehat w_j\widehat g_j^T+
\widehat g_j\widehat w_j^T}
{h_jp_jr_jD_j^{(1)}}
+\left\{\frac1{D_j^{(1)}}
+\frac{p_j+1}{h_jp_j(D_j^{(1)})^2}\right\}
\widehat w_j\widehat w_j^T
\right\}R_j^T.
\label{eq:first-updated-B}
\end{align}

To state the odd-iteration step length without expanding the two eigenvalues
into a long expression, diagonalize
$H_{2j+1}=B_{2j+1}^{-1}$ as prescribed above.  Denote the resulting oriented
eigenbasis and normalized coordinates by
$R_j^{(1)},G_j^{(1)},p_j^{(1)},h_j^{(1)},r_j^{(1)}$; thus
\begin{equation}\label{eq:odd-iteration-data}
\begin{aligned}
g_{2j+1}
&=G_j^{(1)}R_j^{(1)}
\binom{1}{p_j^{(1)}r_j^{(1)}},\\
B_{2j+1}
&=R_j^{(1)}
\begin{pmatrix}
\bigl(h_j^{(1)}p_j^{(1)}(r_j^{(1)})^2\bigr)^{-1}&0\\[1mm]
0&(h_j^{(1)})^{-1}
\end{pmatrix}(R_j^{(1)})^T,\\
\alpha_{2j+1}
&=\frac{p_j^{(1)}+1}
{3h_j^{(1)}p_j^{(1)}D_j^{(2)}},
\qquad
D_j^{(2)}=1-4\epsilon_jr_j^{(1)}+(r_j^{(1)})^2.
\end{aligned}
\end{equation}
Here the secant matrices in the original coordinates are
$R_jA_1(\epsilon_j)R_j^T$ and
$R_j^{(1)}A_2(\epsilon_j)(R_j^{(1)})^T$, respectively.  All formulas in
\eqref{eq:even-iteration-data}--\eqref{eq:odd-iteration-data} are exact.  They
follow by applying the entrywise update formulas
\eqref{eq:appendix-H-entries} and \eqref{eq:appendix-B-update}, followed by
the parameter-recovery formulas \eqref{eq:parameter-recovery}; see
Appendix~\ref{app:algebra}.  Thus the odd-iteration quantities are determined
explicitly by the even-iteration data, although writing the two eigenvalues
out in full would obscure the two-step recurrence.

\subsection{A local invariant graph}

The two-step recurrence is singular at $\epsilon=0$ in the original
variables.  After the explicit powers of $\epsilon$ are cancelled, its
derivative has one eigenvalue equal to one and two eigenvalues with modulus
less than one.  We use the following standard invariant-graph result for a
discrete system.  In the application, the scalar center variable is
$u=\epsilon$, and the stable variable is
$z=(p-2,h-1)\in\R^2$.  The auxiliary system has one center variable and
two stable variables, even though the optimization problem itself is
two-dimensional.  We state the result for a general stable dimension $d_s$;
it also applies to a noninvertible map \cite{KarydasSchinas1992}.

\begin{lemma}[Local invariant graph]\label{lem:graph-transform}
Let $\nu\ge2$, let $d_s\ge1$ be an integer, and let $F=(F_c,F_s)$ be a
$C^\nu$ map from a neighborhood of $(0,0)\in\R\times\R^{d_s}$ into
$\R\times\R^{d_s}$.  Suppose that, for some
$L\in\R^{d_s\times d_s}$,
\[
F(0,0)=(0,0),\qquad
DF(0,0)=\begin{pmatrix}1&0\\0&L\end{pmatrix},
\qquad \rho(L)<1.
\]
Then, after restricting the neighborhood, there is a $C^\nu$ function
$z=\zeta(u)$ with $\zeta(0)=D\zeta(0)=0$ whose graph is locally forward
invariant.  The conclusion does not require $L$, or the full derivative
$DF(0,0)$, to be invertible.
\end{lemma}

\begin{proof}
Choose an equivalent norm in the $z$-space for which
$\norm{L}<\lambda<1$.  Write $\mathbb B_b$ and
$\overline{\mathbb B}_b$ for the open and closed balls of radius $b$ centered
at the origin in this norm, and fix a small Lipschitz bound
$q_{\rm lip}>0$.
Multiply the nonlinear part of $F$ by a smooth cutoff.  The cutoff equals one
near the origin and vanishes near the boundary of
$[-a,a]\times\overline{\mathbb B}_b$.  At $u=\pm a$, the modified
$u$-component is the identity.  The linear map $L$ sends
$\overline{\mathbb B}_b$ strictly inside itself.  By
reducing $a$ and $b$, we can make the modified map send the cylinder into
itself and satisfy the following bounds for any fixed small $\delta>0$:
\begin{align*}
|D_uF_c-1|+\norm{D_zF_c}&\le\delta,\\
\norm{D_uF_s}+\norm{D_zF_s-L}&\le\delta.
\end{align*}
Let $\mathcal X$ be the complete metric space of graphs
$\zeta:[-a,a]\to\overline{\mathbb B}_b$ with $\zeta(0)=0$ and
$\operatorname{Lip}(\zeta)\le q_{\rm lip}$.  We use the uniform norm.  The
boundary choice ensures that the $u$-component of each graph maps $[-a,a]$
onto itself.  Its image also remains in the cylinder.

For $\zeta\in\mathcal X$, set
$\Theta_\zeta(u)=F_c(u,\zeta(u))$.  Its difference quotients lie between
$1-\delta(1+q_{\rm lip})$ and $1+\delta(1+q_{\rm lip})$.  The same bounds hold
for its derivative wherever that derivative exists.  Since
$\Theta_\zeta(\pm a)=\pm a$, this map is a strictly increasing bijection of
$[-a,a]$.  Its inverse is Lipschitz continuous.
Define the graph transform by
\[
(\mathcal T\zeta)(\bar u)
=F_s(u,\zeta(u)),
\qquad u=\Theta_\zeta^{-1}(\bar u).
\]
The derivative bounds show that $\mathcal T\mathcal X\subset\mathcal X$.  Its
uniform contraction factor is $\lambda+O(\delta)<1$.  The map $\mathcal T$
has a unique fixed graph on the cutoff cylinder.  On the smaller
region where the cutoff equals one, that graph is forward invariant for the
original map $F$.

We next prove the stated regularity.  For each
$1\le r\le\nu$, shrink the cylinder until
\begin{equation}\label{eq:graph-bunching}
(\lambda+O(\delta))
\{1-\delta(1+q_{\rm lip})\}^{-r}<1.
\end{equation}
At stage $r$, consider the closed set of $C^r$ graphs whose derivatives through
order $r-1$ satisfy the bounds from the earlier stages.  Use a norm that also
controls the $r$th derivative.  The chain rule shows that the transformed
$r$th derivative is affine in $D^r\zeta$.  Its linear coefficient is bounded
by the left-hand side of \eqref{eq:graph-bunching}; all other terms use only
lower derivatives.  Condition \eqref{eq:graph-bunching} makes the transform a
contraction at stage $r$.  Induction from $r=1$ to $r=\nu$ gives a $C^\nu$
fixed graph.  Uniqueness in the uniform norm identifies it with the Lipschitz
graph constructed above.  Differentiating the invariance equation at the
origin gives $D\zeta(0)=L D\zeta(0)$.  We have $D\zeta(0)=0$ because
$1\notin\sigma(L)$.  This is the standard graph-transform proof of the
discrete center manifold theorem in $C^\nu$.  Only the scalar map
$\Theta_\zeta$ is inverted.  Neither $L$ nor $DF(0,0)$ must be invertible, so
a zero eigenvalue of $L$ causes no difficulty.
\end{proof}

\subsection{Analytic extension and an invariant center manifold}

Lemma~\ref{lem:graph-transform} can be applied only after we remove the
singularity at $\epsilon=0$ and find the spectrum of the linearized recurrence.
Let
$\mathcal F(\epsilon,p,h)=(\epsilon_+,p_+,h_+)$ denote the map obtained by
applying the two updates in \eqref{eq:alternating-secant-data}.  We then write
the updated state in the form \eqref{eq:coordinate-representation}.  The next
lemma verifies the two requirements for $\mathcal F$.  It also gives the
leading terms of the map on the center manifold.  These terms determine the
slow change of the small eigenvalue.

\begin{lemma}[Rescaled recurrence and invariant center manifold]
\label{lem:center-manifold}
After cancellation of the explicit singular factors, $\mathcal F$ has a real
analytic extension near $(0,2,1)$.  Its derivative there has one eigenvalue
equal to $1$ and two other eigenvalues
$-1/9$ and $0$, so it has a
locally invariant one-dimensional $C^7$ center manifold
\[
p=p(\epsilon),\qquad h=h(\epsilon),
\]
on which
\begin{align}
p(\epsilon)&=2+\frac{198}{5}\epsilon^3
-\frac95\epsilon^4+O(\epsilon^5),
&h(\epsilon)&=1+8\epsilon^3+O(\epsilon^5),
\label{eq:center-curve}\\
\epsilon_+&=\epsilon-\frac32\epsilon^4
+\frac54\epsilon^5+O(\epsilon^6).
\label{eq:epsilon-map}
\end{align}
\end{lemma}

\begin{proof}
All operations in one DFP iteration are rational in the matrix entries and the
secant pair.  At the limiting point in the rescaled variables, every remaining
denominator is positive after the explicit powers of $\epsilon$ are cancelled.
The two eigenvalues of the limiting matrix are $0$ and $1$.  Its oriented
eigenbasis depends analytically on the entries.  Direct cancellation
gives the stated analytic extension.  If $\widehat r_+$ denotes the value of
$r$ after two
iterations, then $\widehat r_+/\epsilon^2$ has a positive analytic extension
with value one.  We define
\[
\epsilon_+=\epsilon
\sqrt{\widehat r_+/\epsilon^2}.
\]
For $\epsilon>0$, this agrees with the positive square root of
$\widehat r_+$.  The factor $\epsilon$ makes the three-variable recurrence
analytic across $\epsilon=0$.

To verify the cancellation explicitly, write the state after iteration
$i\in\{1,2\}$ in its new oriented eigenbasis as
\[
\lambda_-^{(i)}=\epsilon^4L_i,\qquad
\lambda_+^{(i)}=\mathcal H_i,\qquad
R_i^Tg_i
=G\binom{\mathcal G_i}{\epsilon^2U_i},\qquad
r_i=\epsilon^2\mathcal R_i.
\]
Direct substitution in \eqref{eq:abstract-step} and the DFP determinant
identity
\[
\det H_+=\det H\,\frac{s^Ty}{y^THy}
\]
show that all displayed factors are analytic.  Their values at
$\epsilon=0$ are
\[
\begin{array}{c|cccccc}
i&L_i&\mathcal H_i&\mathcal G_i&U_i&\mathcal R_i&p_i\\ \hline
1&2&1&1&1&2&1/2\\
2&2&1&1&2&1&2.
\end{array}
\]
Indeed,
\[
r_i=\frac{\lambda_-^{(i)}}
{\lambda_+^{(i)}
 ((R_i^Tg_i)_2/(R_i^Tg_i)_1)}
=\epsilon^2\frac{L_i\mathcal G_i}{\mathcal H_iU_i}.
\]
The four expressions used in each iteration are
\[
g^THg,\qquad (Hg)^TA(Hg),\qquad s^TAs,\qquad (As)^TH(As).
\]
Each equals $\epsilon^4$ times a factor that is analytic and positive at
$\epsilon=0$.
For the first iteration these four limiting factors are $(6,4,4,4)$; for the
second they are $(3,1,1,1)$.  For example, before the first update,
\begin{align*}
g^THg&=hp(p+1)\epsilon^4,\\
(Hg)^TA_1(Hg)&=h^2p^2\epsilon^4
 (1+2\epsilon^3+\epsilon^4).
\end{align*}
It follows that every step length, DFP denominator, spectral coordinate, and
quotient in \eqref{eq:coordinate-representation} has an analytic extension
across $\epsilon=0$.  Also,
$\widehat r_+/\epsilon^2=\mathcal R_2\to1$, as required in the definition of
$\epsilon_+$ above.

The cancellation also gives explicit formulas for the limiting recurrence.
If $p$ and $h$ are left independent while $\epsilon\to0$, then
\begin{align}
\left.\frac{r_+}{\epsilon^2}\right|_{\epsilon=0}
&=\frac{9hp(p+1)}{2\{9hp+(p+1)^2\}},
\label{eq:limiting-radius}\\
\left.p_+\right|_{\epsilon=0}
&=\frac{4\{9hp+(p+1)^2\}^2}
{81hp(p+1)^2},
&\left.h_+\right|_{\epsilon=0}&=1.
\label{eq:limiting-ph}
\end{align}
At $(p,h)=(2,1)$, the factor $r_+/\epsilon^2$ equals one.  The limiting
recurrence also fixes $(2,1)$.  Differentiating
\eqref{eq:limiting-radius}--\eqref{eq:limiting-ph}
gives
\begin{equation}\label{eq:limiting-jacobian}
D(\epsilon_+,p_+,h_+)(0,2,1)
=\begin{pmatrix}
1&0&0\\
0&-1/9&2/3\\
0&0&0
\end{pmatrix}.
\end{equation}
The derivative has the stated three eigenvalues.

To determine the required coefficients, treat $r$ and $b$ as independent
small variables and set
\[
p=2+Pb r,\qquad h=1+Jb r.
\]
Substitute this form into the two DFP updates and use the quadratic formula for
the two simple eigenvalues.  This gives the following two-variable expansions.
Here $C=x-g$, and $e$ is a unit eigenvector for the smaller eigenvalue in the
original coordinates.  The scalar $\phi\in\R$ is its angle, chosen continuously
along the sequence:
\begin{align*}
r_+-r
 &=\frac{b(6J+5P-300)}{18}r^2+O(r^3),\\
p_+-2
 &=\frac{b(6J-P+348)}9r+O(r^2),\\
h_+-1&=8br+O(r^2),\\
\frac{G_+}{G}-1
 &=\frac{b^2(24J-4P+384)-117}{18}r^2+O(r^3),\\
\phi_+-\phi&=-3r+O(r^2),\\
e^T(C_+-C)
 &=-\frac{2b^2(6J-P+96)}9Gr^2+O(G|b|r^3).
\end{align*}
The remainders are uniform for bounded $(b,P,J)$ in a fixed neighborhood.
The leading recurrence for $(P,J)$ is
\begin{equation}\label{eq:transverse-map}
P_+=\frac{6J-P+348}{9}+o(1),
\qquad J_+=8+o(1).
\end{equation}
Its linear part is
\[
L=\begin{pmatrix}-1/9&2/3\\0&0\end{pmatrix},
\]
and its fixed point is
\[
(P,J)=\left(\frac{198}{5},8\right).
\]
Lemma~\ref{lem:graph-transform} applies with $\nu=7$.  One eigenvalue is $1$,
and the other two, $-1/9$ and $0$, lie strictly inside the unit disk.  The
analytic extension gives all required derivatives.  The zero eigenvalue
belongs to the matrix $L$, so the lemma gives a $C^7$ invariant graph.
Because the graph is tangent to the $\epsilon$-axis, we initially know only
$(p-2,h-1)=O(\epsilon^2)$, so the absence of a quadratic term must be checked.
Writing $z=(p-2,h-1)^T$, the exact expansion in
Appendix~\ref{app:algebra} has the form
\[
z_+=Lz+O(\epsilon^3+|\epsilon|\norm{z}+\norm{z}^2),
\qquad \epsilon_+=\epsilon+O(\epsilon^4+|\epsilon|\norm{z}).
\]
If the invariant graph begins with $z=z_2\epsilon^2+O(\epsilon^3)$, its
order-$\epsilon^2$ invariance equation is $z_2=Lz_2$.  Since
$1\notin\sigma(L)$, we have $z_2=0$.  We may write
\[
p=2+P_3\epsilon^3+P_4\epsilon^4+O(\epsilon^5),
\qquad
h=1+H_3\epsilon^3+H_4\epsilon^4+O(\epsilon^5).
\]
The $p$- and $h$-components of the graph-invariance equation give
\begin{equation}\label{eq:graph-coefficient-system}
3H_3-5P_3+174=0,\qquad 8-H_3=0,\qquad
3H_4-5P_4-9=0,\qquad H_4=0.
\end{equation}
The solution is
\[
(P_3,H_3,P_4,H_4)=\left(\frac{198}{5},8,-\frac95,0\right).
\]
For the $\epsilon$-component, the next two coefficients of
$\epsilon_+-\epsilon$ are
\[
\frac{6H_3+5P_3-300}{36}=-\frac32,
\qquad
\frac{6H_4+5P_4+54}{36}=\frac54.
\]
This proves \eqref{eq:center-curve}--\eqref{eq:epsilon-map}.  The leading
relation also follows from the displayed two-variable expansion:
$r_+=r-3\epsilon r^2+O(\epsilon^6)$ when
$b=\epsilon$ and $r=\epsilon^2$.
Equation \eqref{eq:epsilon-map} also gives
$0<\epsilon_+<\epsilon$ for every sufficiently small $\epsilon>0$.
Since $p(0)=2$, $h(0)=1$, and $p(\epsilon)$ and $h(\epsilon)$ are continuous,
all forward iterates satisfy $p,h,\epsilon>0$.
\end{proof}

\subsection{Asymptotic expansions in the original coordinates}

Lemma~\ref{lem:center-manifold} gives the center-manifold expansion in
normalized variables.  We now return to the original coordinates and track
$G_j$, $\phi_j$, and $C_k$ over a complete cycle.  Fix a neighborhood
$0<\epsilon\le\bar\epsilon<1/4$ on which all analytic denominators and the
invariant graph above are defined and the Taylor estimates are uniform.  All
constants below are chosen on this fixed neighborhood.  They do not depend on
the final initial parameter
$\epsilon_0\le\bar\epsilon$.  Choose such an $\epsilon_0>0$, take
$(p_0,h_0)=(p(\epsilon_0),h(\epsilon_0))$, $G_0=1$, and $R_0=I$ in
\eqref{eq:coordinate-representation}, and set $x_0=g_0$, so that
$C_0=x_0-g_0=0$.
These choices define an infinite sequence that satisfies the DFP update
identities.
The variable $C_k=x_k-g_k$ is the minimizer of the quadratic
$z\mapsto\frac12\|z-C_k\|^2$, whose gradient at $x_k$ equals $g_k$.
The tail estimate for $C_k$ will be smaller than the distance from $x_k$ to the
limiting circle.  This difference is used in the interpolation argument.

\begin{lemma}[Asymptotic expansions over two DFP iterations]\label{lem:two-step-expansions}
Let $\phi_j\in\R$ be the angle of the unit eigenvector for the smaller
eigenvalue at iteration $2j$.  Choose this angle continuously along the
sequence.  Let $G_j$ be the corresponding gradient component in
\eqref{eq:coordinate-representation}, and define
\[
C_k=x_k-g_k.
\]
On the invariant center manifold,
\begin{align}
\frac{G_{j+1}}{G_j}
&=1-\frac{13}{2}\epsilon_j^4+O(\epsilon_j^6),
\label{eq:G-map}\\
\phi_{j+1}-\phi_j
&=-3\epsilon_j^2+O(\epsilon_j^4),
\label{eq:phi-map}\\
C_{2j+2}-C_{2j}
&=-\frac{116}{5}G_j\epsilon_j^6e_j
+O(G_j\epsilon_j^7),
\label{eq:C-map}\\
\|C_{2j+1}-C_{2j}\|&=O(G_j\epsilon_j^3).
\label{eq:C-half}
\end{align}
The two consecutive polar-angle increments within a cycle are respectively
\begin{equation}\label{eq:half-angles}
-2\epsilon_j^2+o(\epsilon_j^2),
\qquad -\epsilon_j^2+o(\epsilon_j^2).
\end{equation}
Here $e_j$ is the unit eigenvector corresponding to the smaller eigenvalue,
expressed in the original coordinates.
\end{lemma}

\begin{proof}
Set $P=198/5$ and $J=8$, the leading invariant-manifold coefficients.  Before
putting $b=\epsilon$, the same two-variable expansion used in
\eqref{eq:transverse-map} then gives
\begin{align*}
\frac{G_+}{G}
&=1+\frac{232b^2-65}{10}r^2+O(r^3),\\
\phi_+-\phi&=-3r+O(r^2),\\
e^T(C_+-C)&=-\frac{116}{5}Gb^2r^2\\
&\quad+O(G|b|r^3).
\end{align*}
These remainders are uniform for bounded $b$ in the fixed neighborhood.  The
last remainder is not $o(Gb^2r^2)$ when $b$ and $r$ vary independently.
Along the path used here,
$b=\epsilon$ and $r=\epsilon^2$, it is
$O(G\epsilon^7)=o(G\epsilon^6)$.
The four center manifold coefficients in \eqref{eq:center-curve}, substituted
into the analytic recurrence, sharpen these relations to
\begin{align*}
\frac{G_+}{G}
&=1-\frac{13}{2}\epsilon^4
+\frac{116}{5}\epsilon^6+O(\epsilon^7),\\
\phi_+-\phi&=-3\epsilon^2+O(\epsilon^4),\\
R_j^T(C_{2j+2}-C_{2j})/G_j
&=\binom{-\frac{116}{5}\epsilon^6+O(\epsilon^7)}
{O(\epsilon^8)}.
\end{align*}
Along the single-parameter path $b=\epsilon$, $r=\epsilon^2$, the first-step
formula is
\[
C_{2j+1}-C_{2j}
=G_jR_j\binom{2\epsilon^3+o(\epsilon^3)}
{2\epsilon^5+o(\epsilon^5)}.
\]
These relations prove \eqref{eq:G-map}--\eqref{eq:C-half}.  A direct expansion
in the eigenbases before and after the two updates gives
\eqref{eq:half-angles}.
\end{proof}

\subsection{Scalar asymptotics and the limiting circle}

Lemma~\ref{lem:two-step-expansions} gives local relations from one cycle to the
next.  To use them in the interpolation step, we must show that the recurrence
remains in the chosen neighborhood for all cycles and convert the one-cycle
relations into asymptotic estimates for the full sequence.  The lemma below
does this, introduces the limits $G_\infty$ and $C_\infty$, and identifies the
accumulation set of the iterates.

\begin{lemma}[Scalar asymptotics]\label{lem:scalar-asymptotics}
There are limits $G_\infty>0$ and $C_\infty\in\R^2$ such that
\begin{align}
\epsilon_j&\sim\left(\frac92j\right)^{-1/3},
\label{eq:epsilon-asymptotic}\\
G_j-G_\infty&\sim\frac{13}{3}G_\infty\epsilon_j,
\label{eq:G-tail}\\
\|C_{2j}-C_\infty\|&=O(\epsilon_j^3),
\qquad
\|C_{2j+1}-C_\infty\|=O(\epsilon_j^3).
\label{eq:C-tail}
\end{align}
Also, $\phi_j\to-\infty$, and the accumulation set of the iterates is
the full circle
\begin{equation}\label{eq:limit-circle}
\Gamma=\{C_\infty+G_\infty v:\|v\|=1\}.
\end{equation}
\end{lemma}

\begin{proof}
Equation \eqref{eq:epsilon-map} gives
\[
\epsilon_{j+1}^{-3}-\epsilon_j^{-3}
=\frac92+O(\epsilon_j),
\]
which proves \eqref{eq:epsilon-asymptotic}.  In particular,
\[
\sum_j\epsilon_j^4<\infty,
\qquad
\sum_j\epsilon_j^2=\infty.
\]
The infinite product $\prod_j(G_{j+1}/G_j)$ in \eqref{eq:G-map}
converges to a positive limit $G_\infty$.  Comparing the tails of
\eqref{eq:G-map} with
\eqref{eq:epsilon-map} gives
\[
\sum_{\ell\ge j}\epsilon_\ell^4\sim\frac23\epsilon_j,
\]
which proves \eqref{eq:G-tail}.  The series in \eqref{eq:C-map} is absolutely
convergent, with
\[
\sum_{\ell\ge j}\epsilon_\ell^6=O(\epsilon_j^3).
\]
Combining this with \eqref{eq:C-half} proves \eqref{eq:C-tail}.

Equation \eqref{eq:phi-map} and the divergent sum of $\epsilon_j^2$ imply
$\phi_j\to-\infty$.  The angular increments tend to zero, so every angle
modulo $2\pi$ is approached by the even subsequence.  Direct substitution in
\eqref{eq:abstract-step} and \eqref{eq:coordinate-representation} gives
$\|x_{2j+1}-x_{2j}\|=\|s_{2j}\|=O(\epsilon_j^2)$, so the odd and even
subsequences have the same accumulation set.  Equations \eqref{eq:G-tail} and
\eqref{eq:C-tail} show that this set is the circle in
\eqref{eq:limit-circle}.
\end{proof}

\section{Construction of a uniformly convex objective}
\label{sec:global-objective}

So far, the matrices $A_i(\epsilon)$ have only defined the secant pairs in the
recurrence.  We now construct one objective with the prescribed function
values and gradients at every iterate.  First we show that iterates from
different turns around the limiting circle remain separated.  This allows
local interpolation on pairwise disjoint neighborhoods and preserves uniform
convexity.  For an iterate $x_k$ in cycle $j$, write $r(k)=\epsilon_j^2$.
Convergence to a circle alone does not give disjoint neighborhoods.  Iterates
from different turns could be much closer than $r(k)$.

\subsection{Uniform separation of the iterates}

The estimate must cover neighboring iterates, near returns after one turn, and
pairs separated by many turns.  The following lemma gives one uniform lower
bound for all three cases.  This is the geometric estimate needed for the
interpolation construction.

\begin{lemma}[Uniform separation of the iterates]\label{lem:separation}
After decreasing $\epsilon_0$ if necessary, there is a constant $c_*>0$ such
that, for every $k$,
\begin{equation}\label{eq:separation}
\operatorname{dist}\!\left(
x_k,(\{x_\ell:\ell\ge0\}\cup\Gamma)\setminus\{x_k\}
\right)
\ge c_*r(k).
\end{equation}
\end{lemma}

\begin{proof}
At iteration $2j$, identify $\R^2$ with $\mathbb C$ in the oriented
eigenbasis.  Then
\begin{equation}\label{eq:polar-boundary}
x_{2j}-C_\infty
=G_je^{i\phi_j}\bigl(1+2i\epsilon_j^2+o(\epsilon_j^2)\bigr).
\end{equation}
The polar angles of the iterates are strictly decreasing for small
$\epsilon_0$.  By \eqref{eq:half-angles}, after one uniform reduction of
$\epsilon_0$ the two
consecutive gaps in cycle $j$ lie between $r_j/2$ and $3r_j$.  We fix
$c_\theta=1/2$ and $C_\theta=3$ on this neighborhood.
If needed, enlarge $C_\theta$ so that any block of $N$ cycles with comparable
scales rotates by at most $NC_\theta r_j$.  During one
complete turn, $\epsilon$ changes by only a relative $O(\epsilon_j)$.

For $\sigma\in\{0,1\}$, corresponding to the first and second iterates in a
cycle, respectively, the polar radius satisfies
\begin{equation}\label{eq:within-cycle-radius}
\|x_{2j+\sigma}-C_\infty\|
=G_j+O(\epsilon_j^3)=G_j+o(r_j),
\qquad r_j=\epsilon_j^2.
\end{equation}
We also need a uniform version.  There is a function
$\omega(\eta)\downarrow0$ such that,
whenever $\epsilon_0\le\eta$,
\begin{equation}\label{eq:within-cycle-radius-uniform}
\left|\norm{x_{2j+\sigma}-C_\infty}-G_j\right|
\le \omega(\eta)r_j
\qquad(\sigma=0,1;\ j\ge0).
\end{equation}
Fix two distinct iterates $x_{2j+\sigma}$ and $x_{2\ell+\vartheta}$, where
$\sigma,\vartheta\in\{0,1\}$ and $\ell\ge j$, and normalize the estimates by
$r_j$.  Let $\Theta_{j\sigma,\ell\vartheta}>0$ denote the accumulated
clockwise angle from the earlier iterate to the later one; strict positivity
follows from \eqref{eq:half-angles}.  The following three cases cover all
possibilities.

Fix
\begin{equation}\label{eq:kappa-choice}
\frac1{\sqrt2}<\kappa<1.
\end{equation}
After reducing $\epsilon_0$ if needed, the recurrence makes $\epsilon_j$
strictly decreasing.  For a pair with $\ell\ge j$, we call
the scales \emph{comparable} when
$\kappa\epsilon_j<\epsilon_\ell\le\epsilon_j$; otherwise they are separated.
\emph{Separated scales.}  Suppose
$\epsilon_\ell\le\kappa\epsilon_j$.
Equations \eqref{eq:G-tail} and
\eqref{eq:within-cycle-radius} give
\[
G_j-G_\ell
=\frac{13}{3}G_\infty(\epsilon_j-\epsilon_\ell)(1+o(1))
\ge c\epsilon_j\gg r_j.
\]
The radial component alone gives the required separation.

\emph{Comparable scales with angular separation.}  Suppose the two scales
are comparable and their angular separation modulo $2\pi$ is at least
$r_j/4$.
Both polar radii have a uniform positive
lower bound, and the chord formula gives distance at least $c r_j$.
This case includes adjacent iterates because their angular gaps are bounded
below by a fixed multiple of $r_j$.

\emph{Comparable scales after one or more complete turns.}  It remains to consider
comparable scales with
\begin{equation}\label{eq:unwrapped-return}
\Theta_{j\sigma,\ell\vartheta}=2\pi m+\zeta,
\qquad |\zeta|<r_j/4,
\end{equation}
for an integer $m\ge0$.  The lower bound on consecutive
angular gaps excludes $m=0$: after decreasing $\epsilon_0$, each consecutive
gap is at least $r_t/2$, while comparability and
\eqref{eq:kappa-choice} give
$r_t/2\ge\kappa^2r_j/2>r_j/4$.  Every near return contains at least
one complete turn, so $m\ge1$.

Let $N=\ell-j$.  Here $N\ge1$; by enlarging $C_\theta$ once, the two
within-cycle angular offsets are absorbed into the same $NC_\theta r_j$
bound.  While the scales are comparable, the expansion
\eqref{eq:G-map} gives a fixed $c_G>0$ such that
$G_t-G_{t+1}\ge c_G r_j^2$.  Equation
\eqref{eq:unwrapped-return}, together with the upper bound on the accumulated
angle, gives
\begin{equation}\label{eq:cycle-count}
NC_\theta r_j\ge 2\pi m-r_j/4,
\end{equation}
and gives
\begin{equation}\label{eq:radial-return-gap}
G_j-G_\ell
\ge Nc_G r_j^2
\ge \frac{c_G}{C_\theta}(2\pi m-r_j/4)r_j
\ge c_0m r_j
\end{equation}
for a uniform $c_0>0$.  Reduce $\epsilon_0$ once more so that
$2\omega(\epsilon_0)\le c_0/2$ in
\eqref{eq:within-cycle-radius-uniform}.  The two error terms for positions
within cycles are then less than half of the lower bound in
\eqref{eq:radial-return-gap}, uniformly in the pair and in $m$.  For $m=1$,
summing the leading terms gives the
sharper one-turn
formula
\begin{equation}\label{eq:one-turn-drop}
G_j-G_\ell=\frac{13\pi}{3}G_jr_j(1+o(1)).
\end{equation}
This estimate covers different positions within a cycle, neighboring turns,
and returns after several turns.  It does not assume that the radius
$\|x_k-C_\infty\|$ is monotone at every individual cycle.

Equations \eqref{eq:G-tail}, \eqref{eq:C-tail}, and
\eqref{eq:polar-boundary} imply
\[
\operatorname{dist}(x_k,\Gamma)
=\frac{13}{3}G_\infty\epsilon_j(1+o(1))\gg r(k).
\]
All constants used in the three cases were fixed on the initial neighborhood
$0<\epsilon\le\bar\epsilon$.  They do not depend on the final, sufficiently
small choice of $\epsilon_0$.  This proves
\eqref{eq:separation}.
\end{proof}

\subsection{Local interpolation with disjoint supports}

The separation estimate gives disjoint supports.  Since
$C_k-C_\infty=O(\epsilon_j^3)$ while the support radii are of order
$\epsilon_j^2$, the interpolating functions have small second derivatives.
Define the closed set
\[
\mathcal E=\Gamma\cup\{x_k:k\ge0\}
\]
and put
\begin{equation}\label{eq:dk-rhok}
d_k=\operatorname{dist}(x_k,\mathcal E\setminus\{x_k\}),
\qquad \rho_k=\frac14d_k.
\end{equation}
By Lemma~\ref{lem:separation}, $\rho_k\ge c r(k)$; the adjacent iterate
also gives $\rho_k\le Cr(k)$.  Here $\mathbb B(x,\rho)$ denotes the open
Euclidean ball with center $x$ and radius $\rho$.  The balls
$\mathbb B(x_k,\rho_k)$ are pairwise disjoint and avoid $\Gamma$.  Let
\begin{equation}\label{eq:correction-vector}
a_k=g_k-(x_k-C_\infty)=C_\infty-C_k.
\end{equation}
Lemma~\ref{lem:scalar-asymptotics} gives
\begin{equation}\label{eq:ak-bound}
\|a_k\|=O(\epsilon_j^3)=O(\epsilon_jr(k)).
\end{equation}
Choose one $\chi\in C_c^\infty(\mathbb B(0,1))$ with $\chi\equiv1$ on
$\mathbb B(0,1/3)$, and define
\begin{equation}\label{eq:local-interpolants}
\psi_k(z)=
\chi\!\left(\frac{z-x_k}{\rho_k}\right)a_k^T(z-x_k),
\qquad
\Psi(z)=\sum_{k=0}^\infty\psi_k(z).
\end{equation}

The separation estimate makes the supports pairwise disjoint.
Equation~\eqref{eq:ak-bound} makes the second derivatives small.  The next
lemma shows that $\Psi$ extends across the limiting circle with zero value,
gradient, and Hessian.  It also gives a uniform Hessian bound.  This step turns
the prescribed endpoint data into a single $C^2$ perturbation on $\R^2$.

\begin{lemma}[$C^2$ interpolation with disjoint supports]
\label{lem:disjoint-interpolation}
The function $\Psi$, extended by zero on $\Gamma$, belongs to
$C^2(\R^2)$.  It satisfies
\begin{equation}\label{eq:interpolation-jets}
\Psi(x_k)=0,\qquad \nabla\Psi(x_k)=a_k,
\end{equation}
and
\begin{equation}\label{eq:interpolation-hessian-bound}
\sup_{z\in\R^2}\|\nabla^2\Psi(z)\|
\le K\epsilon_0.
\end{equation}
\end{lemma}

\begin{proof}
Write $u=z-x_k$ and $\xi=u/\rho_k$.  Direct differentiation gives
\begin{align*}
\nabla\psi_k(z)
&=\chi(\xi)a_k
+\rho_k^{-1}(a_k^Tu)\nabla\chi(\xi),\\
\nabla^2\psi_k(z)
&=\rho_k^{-1}\{a_k\nabla\chi(\xi)^T
+\nabla\chi(\xi)a_k^T\}
+\rho_k^{-2}(a_k^Tu)\nabla^2\chi(\xi).
\end{align*}
On the support, $\|u\|\le\rho_k$.  The fixed cutoff together with
\eqref{eq:dk-rhok} and \eqref{eq:ak-bound} implies
\begin{align*}
\|\psi_k\|_\infty&=O(\epsilon_jr(k)^2),\\
\|\nabla\psi_k\|_\infty&=O(\epsilon_jr(k)),\\
\|\nabla^2\psi_k\|_\infty&=O(\epsilon_j).
\end{align*}
At most one summand is nonzero at any point.  The centers have no accumulation
point outside $\Gamma$, and $\rho_k\to0$, so the supports are locally finite on
$\R^2\setminus\Gamma$.  Define on the circle
\begin{equation}\label{eq:zero-circle-jet}
\Psi|_\Gamma=0,\qquad
\nabla\Psi|_\Gamma=0,\qquad
\nabla^2\Psi|_\Gamma=0.
\end{equation}
We verify that these values and derivatives agree with the limits from
$\R^2\setminus\Gamma$.  Equations
\eqref{eq:G-tail} and \eqref{eq:C-tail} give
\[
\operatorname{dist}(x_k,\Gamma)=\Theta(\epsilon_j),
\qquad \rho_k=O(\epsilon_j^2).
\]
Every point $z$ in the $k$th support has
$\operatorname{dist}(z,\Gamma)=\Theta(\epsilon_j)$, and
\[
\frac{|\psi_k(z)|}{\operatorname{dist}(z,\Gamma)^2}=O(\epsilon_j^3),
\qquad
\frac{\|\nabla\psi_k(z)\|}{\operatorname{dist}(z,\Gamma)}
=O(\epsilon_j^2),
\qquad
\|\nabla^2\psi_k(z)\|=O(\epsilon_j).
\]
If $z$ lies outside all supports, the three quantities are already zero.  If
$z$ lies in the $k$th support and tends to any $\zeta\in\Gamma$, then
$\operatorname{dist}(z,\Gamma)\le\norm{z-\zeta}$, and the displayed estimates
imply
\[
\frac{|\Psi(z)-\Psi(\zeta)-\nabla\Psi(\zeta)^T(z-\zeta)|}
     {\norm{z-\zeta}^2}\longrightarrow0,
\qquad
\frac{\norm{\nabla\Psi(z)-\nabla\Psi(\zeta)}}
     {\norm{z-\zeta}}\longrightarrow0,
\]
while $\nabla^2\Psi(z)\to0$.  The value, gradient, and Hessian in
\eqref{eq:zero-circle-jet} are continuous at every point of $\Gamma$, and
$\Psi\in C^2(\R^2)$.  The function, gradient, and Hessian tend to zero at
$\Gamma$.  No additional extension theorem is needed.
The cutoff is constant near each point $x_k$, which gives
\eqref{eq:interpolation-jets}.  The constants in the separation estimate, the upper
bound $\rho_k\le Cr(k)$, and the derivatives of the fixed cutoff were chosen
on the fixed neighborhood $\epsilon\le\bar\epsilon$.  They give one constant
$K$ that does not depend on the later choice of
$\epsilon_0\le\bar\epsilon$.  Disjointness and
$\epsilon_j\le\epsilon_0$ then give
\eqref{eq:interpolation-hessian-bound}.
\end{proof}

\subsection{Definition of the objective}

Lemma~\ref{lem:disjoint-interpolation} gives a $C^2$ correction with a
uniformly small Hessian.  Adding it to the reference quadratic gives the
required objective while preserving the prescribed function values and
gradients.  Set
\begin{equation}\label{eq:objective}
f(z)=\frac12\|z-C_\infty\|^2+\Psi(z).
\end{equation}
After $K$ has been fixed, take
$\epsilon_0\le\min\{\bar\epsilon,(2K)^{-1}\}$, together with the earlier
upper bounds on $\epsilon_0$.  Then the right-hand side of
\eqref{eq:interpolation-hessian-bound} is at most $1/2$, which gives the global
Hessian bounds in Theorem~\ref{thm:main}.  At every iterate,
\begin{equation}\label{eq:endpoint-data}
f(x_k)=\frac12\|x_k-C_\infty\|^2,
\qquad
\nabla f(x_k)=x_k-C_\infty+a_k=g_k.
\end{equation}
This gives $\nabla f(x_k)=g_k$ for every $k$.  In particular, each $y_k$ is the
corresponding gradient difference, and the DFP updates defined above are
unchanged.

\section{Verification of the counterexample}
\label{sec:completion}

For the objective just constructed, $(x_k,g_k,H_k)$ is a classical DFP
sequence.  We verify the line-search inequalities, prove nonconvergence in
every dimension, and compare the example with earlier DFP convergence results.

\subsection{Armijo decrease and the strong Wolfe curvature condition}

It remains to verify that the prescribed positive steps are valid for the
objective just constructed.  We first prove the Armijo inequality
from the endpoint values and then use the exact secant ratio to verify the
strong Wolfe curvature condition.  Let $q_k=-g_k^Ts_k>0$.  Since
$\Psi(x_k)=0$ at every iterate,
\eqref{eq:endpoint-data} gives
\begin{align}
\frac{f(x_k)-f(x_{k+1})}{q_k}
&=1-\frac{\|s_k\|^2}{2q_k}
+\frac{a_k^Ts_k}{q_k}.
\label{eq:armijo-ratio}
\end{align}
The expansions in Section~\ref{sec:alternating-sequence} give
$\|s_k\|=\Theta(r(k))$ and
$q_k=\Theta(r(k)^2)$.  Equations \eqref{eq:ak-bound} and
\eqref{eq:line-ratio}, together with $A_i(\epsilon)\to I$, imply
\[
\frac{a_k^Ts_k}{q_k}=O(\epsilon_j),
\qquad
\frac{\|s_k\|^2}{q_k}=\tau_k+O(\epsilon_j).
\]
The two limits of \eqref{eq:armijo-ratio} are
\[
1-\frac12\frac23=\frac23,
\qquad
1-\frac12\frac13=\frac56.
\]
The error terms are uniform for $0<\epsilon_j\le\epsilon_0$.  Given any
$c_1<2/3$, one final reduction of $\epsilon_0$ makes every ratio in
\eqref{eq:armijo-ratio} at least $c_1$.  This proves the Armijo inequality
throughout the range in \eqref{eq:wolfe-range}.

The curvature calculation is exact and does not depend on objective values:
\[
g_{k+1}^Ts_k=g_k^Ts_k+s_k^Ty_k
=-(1-\tau_k)q_k.
\]
Since $g_k^Ts_k=-q_k$, the ratio in the strong Wolfe curvature condition is
\[
\frac{|g_{k+1}^Ts_k|}{|g_k^Ts_k|}=1-\tau_k
\in\left\{\frac13,\frac23\right\}.
\]
Both iterations in each cycle satisfy \eqref{eq:strong-curvature} exactly when
$c_2\ge2/3$.  They also satisfy the weak Wolfe curvature inequality.  This
proves the claim for every pair in \eqref{eq:wolfe-range}.  The sequence is fixed once
$\epsilon_0$ is chosen; only the required upper bound on $\epsilon_0$ may
depend on $c_1$.  This verification concerns the accepted positive steps; it
does not assert that a particular strong Wolfe line-search algorithm must
generate them.

\subsection{Nonconvergence and extension to higher dimensions}

We now verify the failure of global convergence.  We then embed the
two-dimensional construction in higher dimensions.
Lemma~\ref{lem:scalar-asymptotics} and
\eqref{eq:coordinate-representation} give $\|g_{2j}\|\to G_\infty$.  On the
first iteration of each cycle,
$\|g_{2j+1}-g_{2j}\|=\|y_{2j}\|=O(\epsilon_j^2)$, so
\[
\|g_k\|\longrightarrow G_\infty>0.
\]
All $H_k$ remain positive definite because $s_k^Ty_k=s_k^TA_ks_k>0$ and
DFP preserves positive definiteness.  This completes the construction in
dimension two.  For $n>2$, define
\[
\widehat f(z,w)=f(z)+\frac12\|w\|^2,
\qquad
\widehat H_0=H_0\oplus I,
\qquad w_0=0.
\]
The entire iteration remains in the first two coordinates, so all identities
and inequalities are unchanged.  This completes the detailed verification of Theorem~\ref{thm:main}.

\subsection{Identity initialization under an affine change of variables}

The remaining claim from Section~\ref{sec:statement} concerns the standard
initialization $H_0=I$.  The affine covariance of DFP allows us to normalize
the initial matrix without changing the accepted step lengths or the Wolfe
inequalities.

\begin{proof}[Proof of Corollary~\ref{cor:identity-initialization}]
We obtain the identity-initialized version by applying an affine change of
variables to the two-dimensional example.  Take the example in
Theorem~\ref{thm:main}, and let
$L=H_0^{1/2}$ be the symmetric positive definite square root.  Define
\[
\widetilde f(z)=f(Lz),\quad z_k=L^{-1}x_k,\quad
\widetilde g_k=\nabla\widetilde f(z_k)=L^Tg_k,\quad
\widetilde H_k=L^{-1}H_kL^{-T}.
\]
Then $\widetilde H_0=I$, and the transformed search direction satisfies
\[
-\widetilde H_k\widetilde g_k
=L^{-1}(-H_kg_k).
\]
The same step length $\alpha_k$ gives
$\widetilde s_k=L^{-1}s_k$ and
$\widetilde y_k=\widetilde g_{k+1}-\widetilde g_k=L^Ty_k$.  We also have
\[
\widetilde s_k^T\widetilde y_k=s_k^Ty_k,
\qquad
\widetilde y_k^T\widetilde H_k\widetilde y_k=y_k^TH_ky_k.
\]
Substitution in the DFP update gives
$\widetilde H_{k+1}=L^{-1}H_{k+1}L^{-T}$, so the complete DFP recurrence is
preserved.

The line functions and their directional derivatives are also unchanged:
\[
\widetilde f(z_k+\alpha\widetilde d_k)
=f(x_k+\alpha d_k),\qquad
\nabla\widetilde f(z_k+\alpha\widetilde d_k)^T\widetilde d_k
=\nabla f(x_k+\alpha d_k)^Td_k.
\]
The Armijo and strong Wolfe curvature inequalities hold with the same
constants.  Also,
\[
\nabla^2\widetilde f(z)=L^T\nabla^2f(Lz)L,
\]
and Theorem~\ref{thm:main} gives
\[
\frac12\lambda_{\min}(H_0)I
\preceq\nabla^2\widetilde f(z)
\preceq\frac32\lambda_{\max}(H_0)I.
\]
We may take
$\widetilde m=\frac12\lambda_{\min}(H_0)$ and
$\widetilde M=\frac32\lambda_{\max}(H_0)$.  In addition,
\[
\|\nabla\widetilde f(z_k)\|=\|L^Tg_k\|
\ge \sqrt{\lambda_{\min}(H_0)}\,\|g_k\|,
\]
so the positive gradient limit in Theorem~\ref{thm:main} gives a positive
lower limit.  Applying the same normalization after the orthogonal
direct-sum extension proves the statement for every $n\ge2$.
\end{proof}

\subsection{Relation to known sufficient conditions}
\label{sec:known-results}

The sequence lies outside the additional hypotheses in the known conditional
convergence theorems.  The exact expansions in
Appendix~\ref{app:algebra}, normalized by the gradient component $G_j$, give
\begin{align*}
\norm{g_{2j}}/G_j
 &=1+2\epsilon_j^4+O(\epsilon_j^6),\\
\norm{g_{2j+1}}/G_j
 &=1-2\epsilon_j^3-2\epsilon_j^4+O(\epsilon_j^6),\\
\norm{g_{2j+2}}/G_j
 &=1-\frac92\epsilon_j^4+O(\epsilon_j^6).
\end{align*}
The first iteration decreases the gradient norm by order $\epsilon_j^3$.  The
second increases it by the same order, so eventual monotonicity fails.  The
step lengths satisfy
\[
\norm{s_{2j}}/G_j=2\epsilon_j^2+o(\epsilon_j^2),
\qquad
\norm{s_{2j+1}}/G_j=\epsilon_j^2+o(\epsilon_j^2).
\]
Since $G_j\to G_\infty>0$ and $\sum_j\epsilon_j^2=\infty$, the total step
length $\sum_k\norm{s_k}$ diverges.  Neither Yuan's eventual gradient-norm
monotonicity condition nor his assumption $\sum_k\norm{s_k}<\infty$
\cite{Yuan1995} applies.  The construction also lies outside Xu's theorem:
for the final function, $m=1/2$ and $M=3/2$, but Xu
\cite{Xu1997} assumes, in addition to an eventual monotonicity condition on an
iteration-dependent scalar, that $c_2<(m/M)^3$.  Every parameter pair in
\eqref{eq:wolfe-range} instead satisfies
\[
c_2\ge\frac23>
\left(\frac{1/2}{3/2}\right)^3=\frac1{27},
\]
so its parameter hypothesis is violated.

Powell's uniformly convex theorem \cite{Powell1971} uses exact line search.
His later two-dimensional result \cite{Powell2000} assumes that each line
search finds the first local minimizer.  In both settings the accepted
endpoint is stationary along the search line.  Our steps instead satisfy
\[
g_{k+1}^Td_k=\alpha_k^{-1}g_{k+1}^Ts_k
=-\alpha_k^{-1}(1-\tau_k)q_k\ne0,
\]
so neither exact line-search assumption applies.  The sequence also fails to
converge: its accumulation set is the circle $\Gamma$.  It lies outside
Pu and Yu's conditional result for a convergent iterate sequence
\cite{PuYu1990}.  Pu's later theorem assumes, in addition, a uniformly
positive and bounded Hessian together with Hessian Lipschitz continuity.  In
Pu's notation, the theorem replaces the fixed Wolfe coefficients by
$\rho_k^{\mathrm{Pu}}=\rho\xi_k$ and
$\sigma_k^{\mathrm{Pu}}=\sigma\xi_k$, with, for example,
\[
\xi_k=\min\left\{1,
\frac{\|R_kH_kg_k\|}{\|Q_kg_k\|}\right\},
\]
where the positive semidefinite matrices $R_k$ and positive definite matrices
$Q_k$ satisfy additional uniform spectral conditions
\cite{Pu2002}.  This metric-dependent requirement can make the corresponding
directional-derivative ratio much smaller than the fixed constant in the
standard strong Wolfe condition.  It does not follow from
\eqref{eq:armijo} and \eqref{eq:strong-curvature}.  The result of Liu, Jing,
and Han uses their proposed inexact line search \cite{LiuJingHan2002}.  It
does not cover arbitrary accepted steps that satisfy only the two standard
strong Wolfe inequalities.  Powell's 1976 result concerns BFGS
\cite{Powell1976}.  The global Broyden-class theorem of Byrd, Nocedal, and
Yuan \cite{ByrdNocedalYuan1987} excludes DFP.  The counterexample does not
contradict any of these results.  Each theorem has an extra assumption that
the constructed sequence does not satisfy.

\subsection{Scope of the construction}

The iteration uses the standard inverse-Hessian form of the DFP update.  Its
fixed positive definite matrix $H_0$ comes from
\eqref{eq:coordinate-representation} at a chosen point on the invariant center
manifold.  Once $(c_1,c_2)$ is fixed in \eqref{eq:wolfe-range}, we choose the
initial parameter sufficiently small and then fix the objective
\eqref{eq:objective}.  Its Hessian satisfies the same global bounds at every
point.  The construction uses no exact line search, damping, restart, or
projection.  It also does not assume a uniform bound on the condition numbers
of $H_k$. The matrices in \eqref{eq:alternating-secant-data} generate the prescribed
secant pairs.  We do not claim that they are Hessians of $f$ at the iterates.
After the gradients have been interpolated, the actual secants of $f$ equal
the prescribed secants.  A search segment may cross the support of another
interpolation function.  This does not affect the proof.  The Wolfe tests use
only endpoint values and endpoint directional derivatives.  The fundamental
theorem of calculus, applied to the gradient along the segment, gives the same
prescribed secant vector.

\section{Numerical experiments}
\label{sec:numerics}

The experiments show the geometry of the construction, compare DFP and BFGS
on the same finite objective, and check the main asymptotic identities.  The
main comparisons use two complementary settings.  The first follows the
prescribed DFP sequence at a scale that makes its limiting geometry visible
and checks the recurrence used in the proof.  The second uses a specific
strong Wolfe routine to compare DFP and BFGS on finite objectives with
verified global Hessian bounds.  We finish with a sensitivity check for the
initial trial step used by the line search.  The solid curves in
Figures~\ref{fig:orbit-comparison} and
\ref{fig:finite-gradient-comparison} join computed iterates.  They use no
smoothing, regression, or fitted trajectory.  The dashed circle in
Figure~\ref{fig:orbit-comparison} comes from the proved asymptotic recurrence;
it is not fitted to the plotted points.  None of the experiments is used in
the proof.

\begin{figure}[H]
\centering
\includegraphics[width=0.94\textwidth]{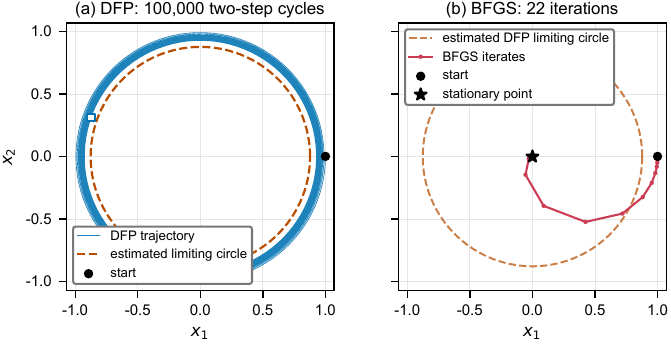}
\caption{Computed DFP and BFGS trajectories for $\epsilon_0=0.03$.
(a) The prescribed DFP sequence over $10^5$ two-step cycles.  All prescribed
steps satisfy the standard strong Wolfe conditions; no claim is made that a
particular line-search code selects the infinite sequence.  The dashed circle
is the asymptotic estimate $G_N\exp(-13\epsilon_N/3)$ centered at $C_N$, and the
open square marks the final iterate.  (b) BFGS on the corresponding finite
$C^2$ interpolant, with the same starting point and initial inverse-Hessian
matrix and a strong Wolfe line search initialized at $\alpha=1$.  The star
marks its stationary point.  The larger value of $\epsilon_0$ is used only to
show the geometry.}
\label{fig:orbit-comparison}
\end{figure}

At $\epsilon_0=0.0025$, which is used for the same-objective comparison below,
the radial drift is only about one percent.  It is almost invisible at journal
size.  For this reason, Figure~\ref{fig:orbit-comparison} uses the larger
display value $\epsilon_0=0.03$.  The recurrence is unchanged.
Figure~\ref{fig:orbit-comparison}(a) follows the two-step construction of
Section~\ref{sec:alternating-sequence} for $10^5$ cycles, or $2\times10^5$
DFP iterations.  At every iteration the code forms $d_k=-H_kg_k$, uses the
positive step length prescribed by the construction, and applies the
unmodified inverse DFP update.  All accepted steps satisfy Armijo and the
standard strong Wolfe curvature inequality with
$(c_1,c_2)=(0.25,0.75)$.  The computed orbit completes $14.5548$ turns and its
gradient norm decreases from $1.0000016$ to $0.9268896$.  If $N=10^5$ denotes
the last cycle, the dashed circle is centered at the computed $C_N$ and has
radius
\[
  \widehat G_\infty=G_N\exp(-13\epsilon_N/3)=0.8770088.
\]
This first-order estimate of the positive limit comes from the asymptotic law.
It is not a least-squares circle.  The open square marks the final DFP iterate.

Figure~\ref{fig:orbit-comparison}(b) uses the same display value, starting
point, and initial inverse-Hessian matrix, but replaces DFP by the classical
inverse BFGS update.  The new update changes the search direction after the
first secant pair, so the DFP endpoints cannot be treated as a BFGS
trajectory.  We instead run BFGS on the corresponding finite $C^2$
interpolant, which can be evaluated directly.  We call the
\texttt{scipy.optimize.line\_search} routine in SciPy~1.17.0
\cite{VirtanenEtAl2020} with
$(c_1,c_2)=(0.25,0.75)$, \texttt{amax=64}, and \texttt{maxiter=40}.  No
previous step estimate is supplied, so the routine tests $\alpha=1$ first and
then uses its bracketing and zoom procedure if that trial is rejected.  The
line search selects the displayed path, and the gradient norm reaches
$3.32\times10^{-12}$ in $22$ iterations.  At the larger value
$\epsilon_0=0.03$, the available bound on the smooth corrections does not
prove global convexity.  We use Figure~\ref{fig:orbit-comparison} only to show
the geometry.

The DFP run in Figure~\ref{fig:orbit-comparison}(a) also provides a numerical
check of the main asymptotic and algebraic identities.  Table~\ref{tab:recurrence-verification}
compares the three leading coefficients with medians over the last $10{,}000$
cycles and reports the largest algebraic residuals.  The DFP update is
unmodified: no damping, restart, projection, or matrix reset is used.  The
dashed circle in Figure~\ref{fig:orbit-comparison} uses the positive-limit
estimate given above.  That estimate comes from the asymptotic formula and is
not fitted to the orbit.  The first normalized coefficient converges more
slowly because, after division by $\epsilon_j^4$, its expansion contains the
correction $(5/4)\epsilon_j+O(\epsilon_j^2)$.

\begin{table}[H]
\caption{Numerical checks for the prescribed DFP sequence.  The first three
numerical values are medians over the final $10{,}000$ cycles; the two
residuals are maxima over $2\times10^5$ steps.}
\label{tab:recurrence-verification}
\centering
\small
\setlength{\tabcolsep}{3.5pt}
\begin{tabular}{@{}p{6.1cm}cc@{}}
\toprule
quantity & theoretical value & numerical value\\
\midrule
$(\epsilon_{j+1}-\epsilon_j)/\epsilon_j^4$
  & $-3/2$ & $-1.48384$\\
$(G_{j+1}/G_j-1)/\epsilon_j^4$
  & $-13/2$ & $-6.49652$\\
$(\phi_{j+1}-\phi_j)/\epsilon_j^2$
  & $-3$ & $-3.00009$\\
ratios in the strong Wolfe condition, first and second iterations
  & $1/3,\ 2/3$ & $0.33333333,\ 0.66666667$\\
$\max_k|s_k^Ty_k/q_k-\tau_k|$
  & $0$ & $0.000000000000276$\\
$\max_k\|H_{k+1}y_k-s_k\|/\|s_k\|$
  & $0$ & $0.000000000000000661$\\
Armijo / strong Wolfe curvature failures
  & $0/0$ & $0/0$\\
\bottomrule
\end{tabular}
\end{table}
\FloatBarrier

We next compare the two methods on one finite objective with verified global
Hessian bounds.  We take a finite sum of the interpolation functions with
$8{,}004$ prescribed endpoints and set $\epsilon_0=0.0025$.  The global
Hessian bound is
\[
0.5532628 I\preceq\nabla^2 f(x)\preceq1.4467372 I
\qquad(x\in\mathbb R^2),
\]
which lies inside the Hessian bounds of Theorem~\ref{thm:main}.  Both methods
start from the same $x_0$ and $H_0$ and use the strong Wolfe routine and
parameter settings described with Figure~\ref{fig:orbit-comparison}(b).  Here
``unit first'' means that the routine first tests $\alpha=1$, not that the
step is fixed at one.  Only the inverse quasi-Newton update differs between
the two runs.

Figure~\ref{fig:finite-gradient-comparison} plots the gradient norms.  The DFP
run stays close to one over $5{,}000$ iterations and ends at $0.9999993653$.
BFGS reaches the stopping tolerance in $32$ iterations, with final gradient
norm $2.22\times10^{-11}$.  Every accepted step in both runs satisfies Armijo
and strong Wolfe.  All curvature denominators are positive, and the
inverse-Hessian approximations remain positive definite.  DFP accepts
$\alpha=1$ in all $5{,}000$ displayed iterations.  BFGS accepts $\alpha=2$ at
iteration $2$ and $\alpha=1$ at every other iteration.
For BFGS, an independent weak-Wolfe implementation accepts exactly the same
steps as the strong-Wolfe routine in this experiment.  The long DFP transient
on a finite objective does not prove infinite nonconvergence.
Theorem~\ref{thm:main} gives that result.

\begin{figure}[H]
\centering
\includegraphics[width=0.56\textwidth]{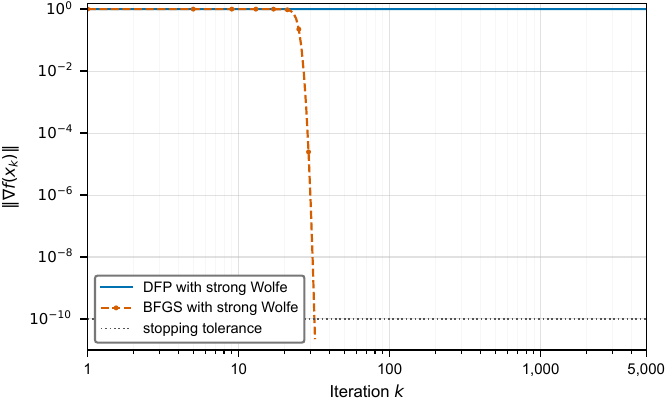}
\caption{Gradient norms for DFP and BFGS on the same finite $C^2$ objective
with $\epsilon_0=0.0025$.  Both methods use the same starting point and initial
inverse-Hessian matrix and a strong Wolfe routine initialized at $\alpha=1$
with $(c_1,c_2)=(0.25,0.75)$.  DFP is shown for $5{,}000$ iterations; BFGS
reaches the displayed tolerance after $32$ iterations.  No fitted values are
used.}
\label{fig:finite-gradient-comparison}
\end{figure}
\FloatBarrier

Table~\ref{tab:finite-comparison} repeats the comparison at three values of
$\epsilon_0$ with verified global Hessian bounds.
The DFP column reports the gradient norm after a fixed budget of $5{,}000$
iterations, not a stopping time.  The BFGS column reports the first iterate
satisfying $\|\nabla f\|\le10^{-10}$.  The accepted weak- and strong-Wolfe
BFGS trajectories agree point for point at all three parameter values.

\begin{table}[!htbp]
\caption{DFP and BFGS with a strong Wolfe line search on finite $C^2$
objectives satisfying the displayed Hessian bounds.  No Armijo or
strong-curvature failures occurred.}
\label{tab:finite-comparison}
\centering
\small
\setlength{\tabcolsep}{3.7pt}
\begin{tabular}{@{}ccccc@{}}
\toprule
$\epsilon_0$ & Hessian eigenvalue bounds &
$\|g_{5000}^{\mathrm{DFP}}\|$ & BFGS iterations &
final $\|g^{\mathrm{BFGS}}\|$\\
\midrule
$0.0010$ & $[0.82140,1.17860]$ & $0.99999998$ & $36$ & $7.38\times10^{-12}$\\
$0.0020$ & $[0.64268,1.35732]$ & $0.99999974$ & $33$ & $1.47\times10^{-11}$\\
$0.0025$ & $[0.55326,1.44674]$ & $0.99999937$ & $32$ & $2.22\times10^{-11}$\\
\bottomrule
\end{tabular}
\end{table}
\FloatBarrier

Finally, we examine how the initial trial step in the line search affects the
finite DFP trajectory.  Let $\rho_k$ be the support radius and define
\[
E_k=\frac{\|x_k^{\mathrm{LS}}-x_k^{\mathrm{ref}}\|}{\rho_k},
\qquad
k_{1/3}=\min\{k:E_k>1/3\}.
\]
Here $x_k^{\mathrm{LS}}$ is produced by the line search, and
$x_k^{\mathrm{ref}}$ is the prescribed DFP endpoint.  We use eight finite
objectives with verified global Hessian bounds and
$0.0005\le\epsilon_0\le0.0025$.  A separate log--log least-squares fit gives
$k_{1/3}=\exp(-2.05218)\epsilon_0^{-1.49730}$ with $R^2=0.999981$.  This is the
only fitted quantity in the numerical section, and it does not appear in
either figure.  For $\epsilon_0=0.001$ and $0.002$, the first indices with
$E_k>1/3$ are $3{,}997$ and $1{,}411$, respectively.  The first indices with
$E_k>1$ are $10{,}141$ and $3{,}441$.  Three independent unit-first Wolfe
implementations generate identical accepted trajectories over $25{,}299$ and
$8{,}945$ tested iterations, respectively.  Every accepted step in the
recorded runs satisfies the standard strong Wolfe conditions.  Thus agreement
among the implementations persists beyond the interval in which the computed
trajectory remains within one support radius of the prescribed reference
sequence.

If the initial trial step is instead based on the previously accepted step,
both independent strong-Wolfe routines accept $\alpha_3=0.841685$ for
$\epsilon_0=0.001$ and $\alpha_3=0.841667$ for $\epsilon_0=0.002$.  In each
case, the accepted point leaves the corresponding support ball.  A unit-first
policy can therefore shadow the prescribed sequence for a long finite
interval, but the Wolfe inequalities alone do not select a unique infinite
sequence.

\FloatBarrier

\section{Conclusion}
\label{sec:conclusion}

The classical weak Wolfe global convergence question for unmodified DFP has a
negative answer.  For every $0<c_1<2/3$ and $2/3\le c_2<1$, we construct a
$C^2$, uniformly convex objective on $\R^2$ with
$\frac12I\preceq\nabla^2f\preceq\frac32I$.  The steps of the classical DFP
sequence satisfy the standard strong Wolfe conditions, but the gradient norms
converge to a positive constant.  The standard strong Wolfe conditions alone
do not guarantee global convergence of DFP.  Between successive
cycle starts, the changes in the gradient norm are summable, but the total
eigenvector rotation is infinite.  Uniform separation of the iterates allows
us to interpolate all endpoint data with one objective.  An affine change of
variables gives a version with $H_0=I$ and problem-dependent Hessian bounds.
Existing DFP convergence results require extra assumptions that this sequence
does not satisfy.  We leave open whether similar examples exist for $c_2<2/3$
or for objectives smoother than $C^2$.

\appendix

\section{Algebraic verification of the two-step DFP recurrence}
\label{app:algebra}

\subsection{One DFP update in matrix entries and spectral coordinates}

This appendix gives the algebraic details used in
Lemmas~\ref{lem:center-manifold} and \ref{lem:two-step-expansions}.  The proof
uses the formulas below.  Work in an orthonormal eigenbasis and write
\[
H=\begin{pmatrix}\ell&0\\0&\eta\end{pmatrix},\qquad
g=\binom{g_1}{g_2},\qquad
A=\begin{pmatrix}a&c\\c&d\end{pmatrix}.
\]
Set
\[
v_1=\ell g_1,\quad v_2=\eta g_2,\quad
w_1=av_1+cv_2,\quad w_2=cv_1+dv_2,
\]
and
\[
\delta=g_1v_1+g_2v_2,\qquad
\beta=v_1w_1+v_2w_2,
\qquad \gamma=\ell w_1^2+\eta w_2^2.
\]
Proposition~\ref{prop:one-step} gives, entry by entry,
\begin{align}
(H_+)_{11}&=\ell-\frac{\ell^2w_1^2}{\gamma}
                    +\frac{v_1^2}{\beta},
&
(H_+)_{12}&=-\frac{\ell\eta w_1w_2}{\gamma}
                    +\frac{v_1v_2}{\beta},
\label{eq:appendix-H-entries}\\
(H_+)_{22}&=\eta-\frac{\eta^2w_2^2}{\gamma}
                    +\frac{v_2^2}{\beta},
&
(g_+)_i&=g_i-\tau\frac{\delta}{\beta}w_i.
\nonumber
\end{align}
These identities also give
\begin{equation}\label{eq:dfp-determinant}
\det H_+=\det H\,\frac{s^Ty}{y^THy},
\end{equation}
either by a two-by-two determinant expansion or by the matrix determinant
lemma.

For the explicit formulas in Section~\ref{sec:alternating-sequence}, it is
also useful to write the same update for $B=H^{-1}$.  Since
$Bs=-\alpha g$ and $y=-\alpha w$, the corresponding Hessian update is
\begin{equation}\label{eq:appendix-B-update}
B_+
=\left(I-\frac{wv^T}{\beta}\right)
B\left(I-\frac{vw^T}{\beta}\right)
+\frac{ww^T}{\beta}
=B-\frac{wg^T+gw^T}{\beta}
+\frac{\delta+\beta}{\beta^2}ww^T.
\end{equation}
This is the inverse of the DFP update; no additional update is introduced.

For the matrices used in the alternating two-step construction, the exact
expressions are explicit.  Start from the coordinate representation
\eqref{eq:coordinate-representation} and let
$A(\mu)=\left(\begin{smallmatrix}1&\mu\\\mu&1\end{smallmatrix}\right)$.
Then
\begin{align}
v&=Ghp r\binom{r}{1},
&w&=Ghp r\binom{r+\mu}{1+\mu r},
\label{eq:generic-vw}\\
\delta&=G^2hp(p+1)r^2,
&\beta&=G^2h^2p^2r^2D_\mu,
\label{eq:generic-delta-beta}\\
\gamma&=G^2h^3p^2r^2
\{pr^2(r+\mu)^2+(1+\mu r)^2\},
&D_\mu&=1+2\mu r+r^2,
\label{eq:generic-gamma}
\end{align}
and
\begin{equation}\label{eq:generic-alpha}
\alpha=\tau\frac{p+1}{hpD_\mu}.
\end{equation}
Equations \eqref{eq:appendix-H-entries} and
\eqref{eq:generic-vw}--\eqref{eq:generic-alpha} are exact rational formulas.
The first iteration uses $(\mu,\tau)=(b,2/3)$; after expressing the updated
quantities in the new eigenbasis, the second uses
$(\mu,\tau)=(-2b,1/3)$.  The coefficients for a complete cycle follow from two
applications of these identities and the spectral formulas below.

For a symmetric matrix
$K=\left(\begin{smallmatrix}a_0&b_0\\b_0&d_0\end{smallmatrix}\right)$,
put
\[
\Delta=\sqrt{(d_0-a_0)^2+4b_0^2},\qquad
\lambda_\pm=\frac{a_0+d_0\pm\Delta}{2}.
\]
Near the limiting matrix $\operatorname{diag}(0,1)$, an analytic eigenvector
associated with $\lambda_-$ is
\[
u_-=
\frac{(d_0-\lambda_-,-b_0)^T}
{\sqrt{(d_0-\lambda_-)^2+b_0^2}},
\qquad u_+=u_-^\perp.
\]
Reverse both signs if needed so that $\gamma_-=u_-^Tg>0$, and write
$\gamma_+=u_+^Tg$.  With the eigenvector choice used in the construction,
$\gamma_+>0$ for all sufficiently small $\epsilon>0$; its leading term is a
positive multiple of $\epsilon^2$.  The denominators below are nonzero.
Comparison with \eqref{eq:coordinate-representation} recovers the
parameters by the exact formulas
\begin{equation}\label{eq:parameter-recovery}
\begin{aligned}
G&=\gamma_-, & h&=\lambda_+,\\
r&=\frac{\lambda_-\gamma_-}{\lambda_+\gamma_+}, &
p&=\frac{\lambda_+\gamma_+^2}{\lambda_-\gamma_-^2}.
\end{aligned}
\end{equation}
Equations \eqref{eq:appendix-H-entries} and
\eqref{eq:parameter-recovery} give the complete calculation for one DFP
iteration and are used for both iterations of each cycle.

\subsection{Two-step expansion and the invariant graph}

To identify terms involving both small variables, replace $\epsilon$ in the
two matrices by an independent variable $b$, treat $r$ as an independent
variable, and set
\[
p=2+Pbr,\qquad h=1+Jbr.
\]
Applying \eqref{eq:appendix-H-entries}, diagonalizing with the preceding
quadratic formula, and recovering the parameters with
\eqref{eq:parameter-recovery} gives
\begin{align}
r_+-r
 &=\frac{b(6J+5P-300)}{18}r^2+O(r^3),
\label{eq:appendix-r}\\
p_+-2
 &=\frac{b(6J-P+348)}9r+O(r^2),
\label{eq:appendix-p}\\
h_+-1&=8br+O(r^2),
\label{eq:appendix-h}\\
\frac{G_+}{G}-1
 &=\frac{b^2(24J-4P+384)-117}{18}r^2+O(r^3),
\label{eq:appendix-G}\\
\phi_+-\phi&=-3r+O(r^2),
\label{eq:appendix-phi}\\
e^T(C_+-C)
 &=-\frac{2b^2(6J-P+96)}9Gr^2+O(G|b|r^3).
\label{eq:appendix-C}
\end{align}
The remainders are uniform for bounded $(b,P,J)$ in the fixed Taylor
neighborhood.  In particular, \eqref{eq:appendix-C} is a joint estimate; only
after setting $b=\epsilon$ and $r=\epsilon^2$ is its remainder
$o(G\epsilon^6)$.

For a general invariant graph tangent to the center subspace, first write
$z=(p-2,h-1)^T=z_2\epsilon^2+O(\epsilon^3)$.  The order-$\epsilon^2$
invariance equation is $z_2=Lz_2$, and $1\notin\sigma(L)$, so $z_2=0$.
For the remaining expansion
\[
p=2+P_3\epsilon^3+P_4\epsilon^4+O(\epsilon^5),\qquad
h=1+H_3\epsilon^3+H_4\epsilon^4+O(\epsilon^5),
\]
the coefficients of $p_+-p(\epsilon_+)$ and
$h_+-h(\epsilon_+)$ are as follows:
\begin{center}
\begin{tabular}{c@{\qquad}c@{\qquad}c}
\toprule
order & $p$-component & $h$-component\\
\midrule
$\epsilon^3$ & $\frac23H_3-\frac{10}{9}P_3+\frac{116}{3}$
              & $8-H_3$\\[2pt]
$\epsilon^4$ & $\frac23H_4-\frac{10}{9}P_4-2$
              & $-H_4$\\
\bottomrule
\end{tabular}
\end{center}
Setting these four entries to zero gives
$(P_3,H_3,P_4,H_4)=(198/5,8,-9/5,0)$, exactly as in
\eqref{eq:graph-coefficient-system}.

\subsection{Coefficients on the invariant center manifold}

The expansions above still contain the coefficients $P_3,H_3,P_4,H_4$.
Restricting them to the invariant center manifold determines the values used
in Lemmas~\ref{lem:center-manifold} and~\ref{lem:two-step-expansions}.
Substituting these four coefficients into the two-step formulas gives the
following expansions.  Where needed, they are normalized by the gradient
component $G_j$ at iteration $2j$:
\begin{align*}
\epsilon_+
 &=\epsilon-\frac32\epsilon^4+\frac54\epsilon^5+O(\epsilon^6),\\
\frac{G_+}{G}
 &=1-\frac{13}{2}\epsilon^4+\frac{116}{5}\epsilon^6
   -\frac{976}{5}\epsilon^7+O(\epsilon^8),\\
\phi_+-\phi
 &=-3\epsilon^2-\frac{196}{5}\epsilon^5
   +\frac{28}{5}\epsilon^6+O(\epsilon^7),\\
R_j^T(C_{2j+2}-C_{2j})/G_j
 &=\binom{-\frac{116}{5}\epsilon^6+\frac{38}{5}\epsilon^7
           +O(\epsilon^8)}
          {-\frac{508}{5}\epsilon^8+O(\epsilon^9)}.
\end{align*}
The two consecutive angle increments are
\begin{align*}
-2\epsilon^2-\frac{122}{5}\epsilon^5
 +\frac{88}{15}\epsilon^6+O(\epsilon^7),\qquad
-\epsilon^2-\frac{104}{5}\epsilon^5
 +\frac{71}{15}\epsilon^6+O(\epsilon^7).
\end{align*}
The corresponding step and gradient norms are
\begin{align*}
\frac{\norm{s_{2j}}}{G_j}
 &=2\epsilon^2+\frac{112}{5}\epsilon^5
   -\frac{11}{5}\epsilon^6+O(\epsilon^7),\\
\frac{\norm{s_{2j+1}}}{G_j}
 &=\epsilon^2+\frac{114}{5}\epsilon^5
   -\frac{49}{10}\epsilon^6+O(\epsilon^7),\\
\frac{\norm{g_{2j}}}{G_j}
 &=1+2\epsilon^4+O(\epsilon^6),\\
\frac{\norm{g_{2j+1}}}{G_j}
 &=1-2\epsilon^3-2\epsilon^4-\frac{112}{5}\epsilon^6
   +O(\epsilon^7),\\
\frac{\norm{g_{2j+2}}}{G_j}
 &=1-\frac92\epsilon^4+\frac{116}{5}\epsilon^6
   +O(\epsilon^7).
\end{align*}
The identity \eqref{eq:line-ratio} also gives the exact ratios
$s^Ty/(-g^Ts)=2/3$ and $1/3$ on the first and second iterations,
respectively.  Every coefficient used in the
asymptotic recurrence, separation argument, comparison with known results, and
Wolfe verification follows from the matrix-entry formulas in this appendix.

\section*{Acknowledgments}

Generative AI assisted manuscript preparation and parts of the mathematical
and computational work. The authors verified all results and assume
responsibility for all content.

\end{document}